\documentclass[11pt]{amsart}

\usepackage[dvips,pdftex]{epsfig}
\usepackage{amssymb}
\usepackage{amsmath}
\usepackage{amscd}
\usepackage{amsthm}
\usepackage[mathscr]{eucal}
\usepackage[all]{xy}
\usepackage{color}

\newtheorem{theorem}{Theorem}
\newtheorem{lemma}{Lemma}
\newtheorem{corollary}{Corollary}

\theoremstyle{definition}
\newtheorem{remark}{Remark}
\newtheorem{definition}{Definition}

\DeclareMathOperator{\rank}{rank}
\DeclareMathOperator{\Jac}{Jac}

\DeclareMathOperator{\Ab}{Ab}

\DeclareMathOperator{\Hom}{Hom}

\DeclareMathOperator{\HH}{H}
\DeclareMathOperator{\K}{K}
\DeclareMathOperator{\TT}{T}

\begin{document}

\title[Abel-Jacobi map under Schiffer variation]{Abel-Jacobi map under Schiffer variation}
\author[Taejung Kim]{Taejung Kim}

\address{Korea Institute for Advanced Study\\
207-43 Cheongyangri-dong\\
Seoul 130-722, Korea}

\thanks{The author would like to thank Jun-Muk Hwang for asking a question of Corollary~\ref{co1} and Scott A. Wolpert for helpful comments about the previous version of this paper and to express his sincere gratefulness to the Korea Institute for Advanced Study for providing him a hospital environment while preparing this manuscript. }

\keywords{Abel-Jacobi map, Elliptic soliton, Hyperelliptic curve, Infinitesimal Hodge structure, Schiffer variation.}

\subjclass[2010]{14D07, 14H42, 14H70,  32G15, 32G20, 35Q51}
\date{\today}
\email{tjkim@kias.re.kr}

\begin{abstract}
We characterize how to vary the Abel-Jacobi map  in terms of Schiffer variation.
From this characterization, we will interpret the relation of hyperellipticity of curves with Schiffer variation and describe the deformation of elliptic solitons under Schiffer variation.
\end{abstract}
\maketitle

\section{Introduction}

The main motivation of this investigation is to see what kind of properties of a Riemann surface $\Gamma$ are changed or preserved under Schiffer variation. The main tool of this paper is a characterization of the variation of the Abel-Jacobi map under Schiffer variation.
From this description, we will see that certain properties of a Riemann surface $\Gamma$ related to the Abel-Jacobi map have natural connections to Schiffer variation. The main ingredient of this description is calculation of an exact variational formula for a basis of holomorphic differentials out of a given variational formula of a period matrix under Schiffer variation. Once we illustrate this, it is an easy matter to determine the variation of the Abel-Jacobi map. The properties associated with the derivatives of the Abel-Jacobi map will be described. For example, the rationality of the first derivative of the Abel-Jacobi map is related to, so-called, elliptic solitons. The vanishing of the second derivative of the Abel-Jacobi map implies the hyperellipticity of a Riemann surface $\Gamma$. Hence, the explicit variational formulae of derivatives of the Abel-Jacobi map under Schiffer variation will show which properties are preserved or not.

In Section~\ref{defle}, we collect basic definitions and known results for the readers' convenience. In Section~\ref{comva}, we reconstruct a basis of holomorphic differentials from a given period matrix under Schiffer variation. In Section~\ref{hysch}, we describe the relation between hyperellipticity and Schiffer variation. As a corollary of this description, we will show that the space $\mathscr{C}_g$ of rank 1 transformations in an infinitesimal variation of Hodge structure associated with a Riemann surface $\Gamma$ has an empty intersection with a tangent space of hyperelliptic locus. In Section~\ref{ellsch}, we will also exhibit an empty intersection between the tangent space of moduli space of elliptic solitons and $\mathscr{C}_g$.

\section{Preliminaries}\label{defle}

\subsection{Hyperellipticity of a Riemann surface}
For more details, see \cite{fa74}:

\begin{definition}
Let $\Gamma$ be a compact Riemann surface of genus $g$. We say that a basis $\{a_1,\dots,a_g,b_1,\dots,b_g\}$ of generators of $\HH_{1}(\Gamma,\mathbb{Z})\cong\mathbb{Z}^{2g}$ is canonical if
$$a_i\circ a_j=b_i\circ b_j=0\text{ and }a_i\circ b_j=\delta_{ij}.$$
\end{definition}

\begin{lemma}\label{lem1}(p.135 in \cite{fa74})
Let $\eta$ be a holomorphic differential on a Riemann surface with  a local expression $\eta=f(\tau)d\tau$ around $p$ and $\omega^{(n)}_{p}$ be a normalized abelian differential of second kind with a local expression
$$\omega^{(n)}_{p}=\frac{d\tau}{\tau^{n+1}}.$$
Then we have
$$\frac{\partial^{n-1}}{\partial\tau^{n-1}}f(0)=\frac{\partial^n}{\partial\tau^n}\int_{p_0}^{p}\eta=\int_{\gamma}\omega^{(n)}_{p}\text{ where}$$
$\tau(p)=0$ and $\gamma$ is the dual cycle associated with $\eta$.
\end{lemma}

\begin{definition}
A Riemann surface $\Gamma$ is hyperelliptic at $p$ if there exists a local coordinate $\tau$ of $\Gamma$ at $p$ such that $\tau^2$ defines a global projection $\Gamma\to\mathbb{P}^1$ of degree 2 ramified at $p$.
\end{definition}

\begin{lemma}\label{lehy1}
A Riemann surface $\Gamma$ is hyperelliptic at $p$ if and only if
$$\frac{\partial^2}{\partial\tau^2}\Ab(\Gamma)|_{\tau=0}=0\text{ where}$$
$\Ab$ is the Abel-Jacobi map $\Ab:\Gamma\to\Jac(\Gamma)$.
\end{lemma}

\begin{proof}
Let $\{\omega_1,\dots,\omega_g\}$ be a basis of normalized holomorphic differentials and $\{b_1,\dots,b_g\}$ be a basis of dual cycles associated with the normalized holomorphic differentials: For $j=1,\dots,g$, we have
$$\begin{aligned}
\frac{\partial^2}{\partial\tau^2}\Ab(\Gamma)_j|_{\tau=0}
&=\frac{\partial^2}{\partial\tau^2}\int_{p_0}^{p}\omega_j\stackrel{\text{Lemma}~\ref{lem1}}{=}\int_{b_j}\omega_{p}^{(2)}\\
&=\int_{b_j}d(\frac{1}{\tau^2})\stackrel{Stokes Theorem}{=}0.
\end{aligned}$$
\end{proof}

\subsection{Schiffer's interior  variation}

Let $\Gamma$ be a compact Riemann surface of genus $g\geq1$. For a given point $p\in\Gamma$, we fix a coordinate neighborhood $(U_p,\tau)$ with $\tau(p)=0$. Let
$$D_p=\tau^{-1}(\{\tau\in\mathbb{C}\mid|\tau|<1\})\text{ and }C_p=\tau^{-1}(\{\tau\in\mathbb{C}\mid|\tau|=1\}).$$
Let $\tau_\epsilon=\tau+\frac{\epsilon}{\tau}$. When $|\epsilon|$ is small enough, $C_\epsilon=\tau_\epsilon(C_p)$ is a simple closed analytic curve, i.e., a Jordan curve in the $\tau_\epsilon$-plane. The Jordan domain with boundary $C_\epsilon$ is denoted by $D_\epsilon$. A new Riemann surface $\Gamma_\epsilon$ with a new complex structure after Schiffer variation of point $p\in \Gamma$ is defined by removing $D_p$ from $\Gamma$ and attaching $D_\epsilon\cup C_\epsilon$ with boundary identification by $\tau_\epsilon$. More precisely, letting $U_1$ be an annular neighborhood of containing $C_p$ for which $C_p\subset U_1\subset U$, we glue $(\Gamma-D_p)\cup U_1$ and $D_\epsilon\cup \tau_\epsilon(U_1)$ by identifying $\tau_\epsilon(U_1)$ and $U_1$ under the mapping $\tau_\epsilon$. Note that a system of charts for $\Gamma_\epsilon$ is given by the original charts from $\Gamma$ on the interiors of $\Gamma-D_p$ and $\tau_\epsilon$ on the interior of $D_\epsilon$ in addition to the following extra coordinate $(h,U_1)$:
$$h(p)=\begin{cases}\tau\text{ on outside of $C$}\\\tau_\epsilon\text{ on inside of $C$}\end{cases}$$
Note that we have an explicit quasi-conformal map $\tau(D_p)$ to $D_\epsilon$,
$$\phi_\epsilon(\tau)=\tau+\epsilon\overline{\tau}=\tau_\epsilon.$$
For more details about Schiffer variation, one can consult \cite{gar75, na85, pa63, sh54}.

\subsection{Infinitesimal Hodge structure}

See \cite{gr84} for details:

\begin{definition}
An infinitesimal variation of Hodge structure of weight n is the set of the following data
$\{\HH_\mathbb{Z},\HH^{p,q},Q,\TT,\delta\}$ such that

\begin{itemize}
\item[(i)]
a complex vector space
$$\HH=\bigoplus_{p+q=n}\HH^{p,q}\text{ and }\HH^{p,q}=\overline{\HH^{q,p}}\text{ where }$$
$\HH_\mathbb{Z}$ is a free $\mathbb{Z}$-module, $-$ is the complex conjugation, and $\HH=\HH_\mathbb{Z}\otimes\mathbb{C}$ with complex linear subspaces $\HH^{p,q}$ for integers $p,q\geq0$
\item[(ii)]
a bilinear form $Q:\HH_\mathbb{Z}\times\HH_\mathbb{Z}\to\mathbb{Z}$
$$\begin{aligned}
Q(\phi,\psi)&=(-1)^nQ(\psi,\phi)\\
Q(\psi,\phi)&=0\text{ for }\psi\in\HH^{p,q},\phi\in\HH^{p',q'},p\ne p'\\
(\sqrt{-1})^{p-q}Q(\psi,\overline{\psi})&>0\text{ for any nonzero }\psi\in\HH^{p,q}
\end{aligned}$$
\item[(iii)]
a finite dimensional complex vector space $\TT$ and a linear map $\delta$
$$\delta=\bigoplus_{p=1}^{n}\delta_p:\TT\to\bigoplus_{p=1}^{n}\Hom(\HH^{p,q},\HH^{p-1,q+1})$$
satisfying
$$\begin{aligned}
\delta_{p-1}(\xi_1)\delta_p(\xi_2)&=\delta_{p-1}(\xi_2)\delta_p(\xi_1)\text{ for }\xi_1,\xi_2\in\TT\\
Q(\delta(\xi)\psi,\phi)+Q(\psi,\delta(\xi)\phi)&=0\text{ for }\psi\in\HH^{p,q},\phi\in\HH^{q+1,p-1},\text{ and }\xi\in\TT.
\end{aligned}$$
\end{itemize}
\end{definition}

A typical example of an infinitesimal variation of Hodge structure of weight $1$, which will be used in this paper is as follows: A finite dimensional complex vector space $\TT$ is $\HH^1(\Gamma,\Theta)$ where $\Gamma$ is an algebraic curve and $\Theta$ is the sheaf of holomorphic vector fields on $\Gamma$. Note that according to the Kodaira-Spencer deformation theory, $\HH^1(\Gamma,\Theta)$ can be regarded as a tangent space $\TT_{[\Gamma]}\mathscr{M}_g$ of $\mathscr{M}_g$ at $[\Gamma]$ where $\mathscr{M}_g$ is the moduli space of curves of genus $g$. On the other hands, a linear map $\delta$ is the differential 
$$\delta:=\phi_\ast:\HH^1(\Gamma,\Theta)\to\Hom(\HH^{1,0},\HH^{0,1})$$
of the period map $\phi:\mathscr{M}_g\to\mathbf{Sp}(g,\mathbb{Z})\backslash\mathbb{H}_g$ where $\mathbb{H}_g$ is the Siegel upper half space. Let $\mathscr{C}_g$ be the space of rank 1 transformations in $\Hom(\HH^{1,0}(\Gamma),\HH^{0,1}(\Gamma))$. Since $\dim_\mathbb{C}\mathscr{M}_g=3g-3$ and $\dim_\mathbb{C}\Hom(\HH^{1,0},\HH^{0,1})=g^2$, 
it is not obvious whether $\delta(\TT_{[\Gamma]}\mathscr{M}_g)$ and $\mathscr{C}_g$ have a nonempty intersection or not. In fact, we have the following theorem.

\begin{theorem}\label{th1}(p.271 in \cite{gr83} and p.56 in \cite{gr84}) For a general algebraic curve $\Gamma$ of genus $\geq5$, the rank one degeneracy locus $\Sigma$ is the bi-canonical image, i.e.,
$$\Sigma:=\{\xi\in\mathbb{P}\HH^1(\Gamma,\Theta)\mid\rank\delta(\xi)\leq1\}=\phi_{2\K_\Gamma}(\Gamma)\text{ where}$$
$\K_\Gamma$ is a canonical bundle of $\Gamma$ and $\phi_{2\K_\Gamma}:\Gamma\to\mathbb{P}^{3g-2}$ is a bi-canonical mapping associated with a linear system $|2\K_\Gamma|$. In general,
$$\phi_{2\K}(\Gamma)\subseteq\Sigma.$$
\end{theorem}

After Theorem~\ref{th1}, we may ask whether $\delta(\TT_{[\Gamma]}\mathscr{H}_g)$ and $\mathscr{C}_g$ have a nonempty intersection or not where $\mathscr{H}_g$ is the space of hyperelliptic curves and $\Gamma$ is a hyperelliptic curve. We will answer this question in Corollary~\ref{co1}. The following remark will be helpful for our discussion later.
\begin{remark}
Any non-zero vector $t_p\in\HH^1(\Gamma,\Theta)$ lying over $\phi_{2\K}(p)\in\mathbb{P}\HH^1(\Gamma,\Theta)$ turns out to be a Schiffer variation. So by Theorem~\ref{th1}, every rank 1 transformation is a Schiffer variation when $g\geq5$. See {\em p.274} in \cite{gr83} for details.
\end{remark}

\section{Holomorphic structure and variation formula}\label{comva}
Let $\{a_1,\dots,a_g,b_1,\dots,b_g\}$ be a canonical basis of a Riemann surface $\Gamma$ and $\{\omega_1,\dots,\omega_g\}$ be a basis of normalized holomorphic differentials with respect to the canonical basis of cycles. we extend this to complete dual basis 
$$\{\omega_1,\dots,\omega_g,\eta_1,\dots,\eta_g\}$$
of smooth $\mathbb{C}$-valued closed differential forms. That is,
$$\int_{a_j} \omega_i=\int_{b_j}\eta_i=\delta_{ij}.$$
We let a period matrix
$$\Pi=\Big(\pi_{ij}(\Gamma)\Big)_{g\times g}:=(\int_{b_i}\omega_j)_{g\times g}.$$
It is well-known that this matrix satisfies $\Im\Pi>0$ and $\Pi^t=\Pi$. When $\Gamma_{\epsilon}^{\ast}$ is a new Riemann surface induced by a Schiffer variation around $p_0$ from $\Gamma$, a variation formula in {\em p.234} of \cite{pa63} for periods is given by
\begin{equation}\label{va1}
\pi_{ij}(\Gamma_{\epsilon}^{\ast})=\pi_{ij}(\Gamma)+\epsilon f_i(p_0)f_j(p_0)+O(\epsilon^2).
\end{equation}
Here $f_i(z)$ is a local expression of $\omega_i$ around $p_0$, i.e., $\omega_i(z)=f_i(z)dz$. Variation formula~\eqref{va1} gives a variational formula for holomorphic differentials:
$$\omega_{j}^{\ast}=\omega_j+\sum_{k=1}^{g}(\epsilon f_j(p_0)f_k(p_0)+O(\epsilon^2))\eta_k.$$
The upshot is to describe the error term $O(\epsilon^2)$ completely to investigate the properties which we will formulate later. In general, we have

\begin{lemma}\label{vaf1}
$$\pi_{ij}(\Gamma_{\epsilon}^{\ast})=\pi_{ij}(\Gamma)+\sum_{n=1}^{\infty}\frac{\epsilon^n}{n!(n-1)!}\sum_{s+m=n-1}\begin{pmatrix}n-1\\m\end{pmatrix}(\frac{d^{n-1+m}}{dt^{n-1+m}}f^{\ast}_{i}(p_0))\frac{d^{s}}{dt^{s}}f_j(p_0).$$
\end{lemma}

\begin{proof}
This formula will be deduced from a modification of a formula in {\em p.234} in \cite{pa63}. Note that the formula in \cite{pa63} is derived by calculating a Taylor expansion ({\em p.231} in \cite{pa63}) of an abelian differential of the third kind up to first order. In order to get the desired formula in the lemma, we calculate the whole series of the Taylor expansion of the abelian differential of the third kind completely. Hence we have
$$\pi_{ij}(\Gamma_{\epsilon}^{\ast})-\pi_{ij}(\Gamma)=\frac{1}{2\pi i}\int_\beta\sum_{n=1}^{\infty}\frac{1}{n!}\frac{\epsilon^n}{t^n} (\frac{d^{n-1}}{dt^{n-1}}f^{\ast}_{i}(t))f_j(t)dt.$$
Here $\beta$ is a curve encompassing $p_0$, i.e., $t(p_0)=0$ and $f^{\ast}_{i}(t)$ is a local expression of the new holomorphic differential $\omega^{\ast}_{i}$ after performing a Schiffer variation around $p_0$, i.e., 
$\omega^{\ast}_{i}(t)=f^{\ast}_{i}(t)dt$. The residue theorem implies
$$\pi_{ij}(\Gamma_{\epsilon}^{\ast})-\pi_{ij}(\Gamma)
=\sum_{n=1}^{\infty}\frac{\epsilon^n}{n!(n-1)!}\sum_{s+m=n-1}\begin{pmatrix}n-1\\m\end{pmatrix}(\frac{d^{n-1+m}}{dt^{n-1+m}}f^{\ast}_{i}(p_0))\frac{d^{s}}{dt^{s}}f_j(p_0).$$
\end{proof}

From Lemma~\ref{vaf1}, we obtain the variation formula for a holomorphic differential:

\begin{lemma}
$$\omega_{j}^{\ast}=\omega_j+\sum_{k=1}^{g}\sum_{n=1}^{\infty}\frac{\epsilon^n}{n!(n-1)!}\sum_{s+m=n-1}\begin{pmatrix}n-1\\m\end{pmatrix}(\frac{d^{n-1+m}}{dt^{n-1+m}}f^{\ast}_{j}(p_0))\frac{d^{s}}{dt^{s}}f_k(p_0)\eta_k.$$
\end{lemma}

\section{Hyperelliptic curve under Schiffer variation}\label{hysch}

Let us look at the derivative of the Abel-Jacobi map at $p\ne p_0$. Note that a local coordinate $\tau$ around $p$ under Schiffer variation at $p_0$ is still the same as the original one $\tau$:

$$\begin{aligned}
\frac{\partial}{\partial\tau}\Ab(\Gamma_{\epsilon}^{\ast})_j&=\frac{\partial}{\partial\tau}\int_{p_0}^{p}\omega_{j}^{\ast}\\
&=\frac{\partial}{\partial\tau}\Big(\int_{p_0}^{p}\omega_j+\sum_{k=1}^{g}(\epsilon f_j(p_0)f_k(p_0)+O(\epsilon^2))\int_{p_0}^{p}\eta_k\Big).\end{aligned}$$
Let a local expression of $\omega_j$ at $p$ be $\omega_j(\tau)=g_j(\tau)d\tau$ where $\tau(p)=0$ and
$$\frac{\partial}{\partial\tau}\int_{p_0}^{p}\eta_{k}:=h_k(\tau).$$
Then we have
\begin{equation}\label{ell1}
\begin{aligned}
&\frac{\partial}{\partial\tau}\Ab(\Gamma_{\epsilon}^{\ast})_j|_{\tau=0}\\
&=g_j(0)+\sum_{k=1}^{g}\sum_{n=1}^{\infty}\frac{\epsilon^n}{n!(n-1)!}\sum_{s+m=n-1}\begin{pmatrix}n-1\\m\end{pmatrix}(\frac{d^{n-1+m}}{dt^{n-1+m}}f^{\ast}_{j}(p_0))\frac{d^{s}}{dt^{s}}f_k(p_0)h_k(0).
\end{aligned}
\end{equation}
Moreover, we also see that

\begin{equation}\label{hy1}
\begin{aligned}
&\frac{\partial^2}{\partial\tau^2}\Ab(\Gamma_{\epsilon}^{\ast})_j|_{\tau=0}\\
&=g'_j(0)+\sum_{k=1}^{g}\sum_{n=1}^{\infty}\frac{\epsilon^n}{n!(n-1)!}\sum_{s+m=n-1}\begin{pmatrix}n-1\\m\end{pmatrix}(\frac{d^{n-1+m}}{dt^{n-1+m}}f^{\ast}_{j}(p_0))\frac{d^{s}}{dt^{s}}f_k(p_0)h'_k(0).
\end{aligned}
\end{equation}
We are now ready to investigate the hyperellipticity of $\Gamma_{\epsilon}^{\ast}$. From Lemma~\ref{lehy1}, to show the invariance of hyperellipticity under Schiffer variation, we need to prove
$\frac{\partial^2}{\partial\tau^2}\Ab(\Gamma_{\epsilon}^{\ast})=0$ for any $\epsilon$. If we assume $\Gamma$ to be a hyperelliptic curve, i.e., $\frac{\partial^2}{\partial\tau^2}\Ab(\Gamma)_j=g'_j(0)=0$, it suffices to show that

\begin{equation}\label{hy3}
0=\sum_{k=1}^{g}\sum_{n=1}^{\infty}\frac{\epsilon^n}{n!(n-1)!}\sum_{s+m=n-1}\begin{pmatrix}n-1\\m\end{pmatrix}(\frac{d^{n-1+m}}{dt^{n-1+m}}f^{\ast}_{j}(p_0))\frac{d^{s}}{dt^{s}}f_k(p_0)h'_k(p).
\end{equation}
Since

\begin{equation}\label{hy4}
\begin{aligned}
&\sum_{k=1}^{g}\sum_{n=1}^{\infty}\frac{\epsilon^n}{n!(n-1)!}\sum_{s+m=n-1}\begin{pmatrix}n-1\\m\end{pmatrix}(\frac{d^{n-1+m}}{dt^{n-1+m}}f^{\ast}_{j}(p_0))\frac{d^{s}}{dt^{s}}f_k(p_0)h'_k(p)\\
&=\sum_{n=1}^{\infty}\frac{\epsilon^n}{n!(n-1)!}\sum_{s+m=n-1}\begin{pmatrix}n-1\\m\end{pmatrix}(\frac{d^{n-1+m}}{dt^{n-1+m}}f^{\ast}_{j}(p_0))\sum_{k=1}^{g}\frac{d^{s}}{dt^{s}}f_k(p_0)h'_k(p),
\end{aligned}
\end{equation}
we see that Equation~\eqref{hy4} is identically zero for any $\epsilon$ if and only if
\begin{equation}\label{hyp2}
\sum_{k=1}^{g}\frac{d^i}{dt^i}f_k(p_0)h'_k(p)=0\text{ for }i=0,\dots,\infty.
\end{equation}
That is, the hyperellipticity is preserved if Equation~\eqref{hyp2} holds. There are two possibilities for this case. One is the existence of a point $p_0$ such that Equation~\eqref{hyp2} is satisfied: Let $\omega_1,\dots,\omega_g$ be a basis of holomorphic differentials of a compact Riemann surface $\Gamma$. Around $p_0$ with a local coordinate $t(p_0)=0$, locally we may let a holomorphic differential $\omega_j(t)=f_j(t)dt$. Using a Taylor expansion, we may have
$$\omega_j=\sum_{i=0}^{\infty}\frac{d^i}{dt^i}f_j|_{t=0}t^i dt.$$
Suppose there exists a \textit{nonzero} vector $(c_1,\dots,c_g)$ such that
\begin{equation}
\sum_{k=1}^{g}c_k\frac{d^i}{dt^i}f_k(0)=0\text{ for }i=0,\dots,\infty.
\end{equation}
Clearly, it is equivalent to supposing
\begin{equation}
\sum_{k=1}^{g}c_k\omega_k=0\text{ in the neighborhood of }p_0.
\end{equation}
Since it is a holomorphic form, this is globally true. But since $\{\omega_1\dots,\omega_g\}$ is a basis, we can not have such a
\textit{nonzero} vector $(c_1,\dots,c_g)$.

The other possibility of Equation~\eqref{hyp2} satisfied is that there is $p\in\Gamma$ such that there exists a basis $\{\eta_1,\dots,\eta_g\}$ such that
$$h'_k(p)=0\text{ for }k=1,\dots g.$$
Since $\Gamma$ is assumed to be a hyperelliptic curve, i.e., $\frac{\partial^2}{\partial\tau^2}\Ab(\Gamma)_j=g'_j(0)=0$, we can conclude that if such basis exists for $p\in\Gamma$, then any Schiffer deformation of a hyperelliptic curve must be a hyperelliptic curve. But Theorem~\ref{th1} gives a contradiction if $g>2$, since the bi-canonical curve generates the whole $3g-3$ linear space $\TT_{[\Gamma]}\mathscr{M}_g$. Hence this forces $g=2$.

\begin{theorem}
A Schiffer variation can not preserve a hyperelliptic property unless $g=2$.
\end{theorem}

\begin{proof}
When the genus is $>2$, from what we have shown, we conclude that Schiffer variation at $p_0$ can not preserve the hyperellipticity at $p$ when $p\ne p_0$. Note that the hyperellipticity of a Riemann surface is independent of choosing a point. Since any Weierstrass points of a hyperelliptic curve can serve as another hyperelliptic point, we can also conclude that Schiffer variation at $p$ can not preserve the hyperellipticity.

\end{proof}

\begin{corollary}\label{co1}
Let $\mathscr{C}_g$ be the space of rank 1 transformations and $\mathscr{H}_g$ be moduli space of hyperelliptic curves. Then
$$\delta(\TT_{[\Gamma]}\mathscr{H}_g)\bigcap\mathscr{C}_g=\emptyset\text{ for }g\geq5.$$
\end{corollary}

\section{Elliptic soliton under Schiffer variation}\label{ellsch}

\subsection{Elliptic soliton}
A general finite gap solution of the K-DV equation
$$\frac{\partial}{\partial t}u+\frac{1}{4}(6u\frac{\partial}{\partial x}u-\frac{\partial^3}{\partial x^3}u)=0$$
or the K-P equation
$$\frac{3}{4}\frac{\partial^2}{\partial y^2}u+ \frac{\partial}{\partial x}\Big(\frac{\partial}{\partial t}u+\frac{1}{4}(6u\frac{\partial}{\partial x}u-\frac{\partial^3}{\partial x^3}u)\Big)=0$$
is written as a theta function $\theta$ related to a compact Riemann surface of genus $g$ explicitly by, so-called, \textit{Its-Mateev formula} by the work of Russian school:
$$u(x,y,t)=2\frac{\partial^2}{\partial x^2}\ln \theta(\mathbf{U}x+\mathbf{V}y+\mathbf{W}t+z_0)+\text{constant where}$$
$\mathbf{U},\mathbf{V}$, $\mathbf{W}$ are some vectors in $\mathbb{C}^g$. In particular, for a given base point $p\in\Gamma$ with a local coordinate $\tau$, it turns out that
$$\mathbf{U}=\frac{\partial}{\partial \tau}\Ab(\Gamma)|_{\tau=0}.$$
Analytically, an elliptic soliton, also known as an \textit{elliptic potential}, is a solution $u(x,y,t)$ written as a combination of elliptic functions by reduction of a theta function to elliptic functions. A particular family of elliptic solitons for the K-DV equation is given by H. Airault, H. P. McKean, and J. Moser in \cite{amm77} and for the K-P equation by I. Krichever in \cite{kri80} using the dynamics of the Calogero-Moser system as follows:

\begin{theorem}\cite{kri80}
If $u(x,y,t)=c+2\sum_{i=1}^{n}\wp(x-x_i(y,t))$ is an elliptic soliton where $c$ is a constant, then $x_i(y,t)$ satisfies the Calogero-Moser system
$$\frac{\partial^2}{\partial y^2}x_i=4\sum_{k\ne i}\wp'(x_i-x_k)\text{ for }i=1,\dots,n\text{ where }$$
$\wp'(z)$ is the usual derivative of the Weierstrass function $\wp(z)$ with respect to $z$.
\end{theorem}

Note that this is an isospectral deformation of a Lam\'{e} potential $u(x)=g(g+1)\wp(x)$. Finding a new family of elliptic solitons which is non-isospectral deformation of the Lam\'{e} potential had been an open question until the appearance \cite{tv90, ve88} of the investigation by A. Treibich and J.-L. Verdier who introduced a geometric definition of an elliptic soliton  inspired by the work \cite{kri80} of Krichever.

\begin{definition}\cite{tv90}
Let $\pi:(\Gamma,p)\to (E,q)$ be a finite pointed morphism from a compact Riemann surface $\Gamma$ to an elliptic curve $E$. We say that $\Gamma$ is an elliptic soliton if  $\pi^\ast E$ is tangent to $\Ab(\Gamma)$ at $\Ab(p)$.
\end{definition}

Note that it is also called an \textit{elliptic tangential cover}. Notice that $\pi^\ast E$ is a 1-dimensional abelian sub-variety of $\Jac(\Gamma)$. What this means is that $\mathbb{Q}^2\cong\pi^\ast\HH^1(E,\mathbb{Q})\subset\HH^1(\Jac(\Gamma),\mathbb{Q})$ becomes a 1-dimensional complex vector space of $\HH^{1,0}(\Jac(\Gamma))$ after identifying $\HH^{1}(\Jac(\Gamma),\mathbb{R})$ with $\HH^{1,0}(\Jac(\Gamma))$ canonically. From this, it is not hard to see that $\Gamma$ is an elliptic soliton if $\frac{\partial}{\partial\tau}\Ab(\Gamma)(p)$ is a rational vector in $\HH^{1,0}(\Gamma)$.

\subsection{Deformation of elliptic solitons}

Let $\mathscr{E}_g$ be a moduli space of elliptic solitons of genus $g$. The following theorem is well-known:

\begin{theorem}\label{thel}\cite{trei89}
$$\dim_{\mathbb{C}}\mathscr{E}_g=g.$$
Moreover, if we fix an elliptic curve, a reduced moduli space $\widetilde{\mathscr{E}_g}$ for a fixed elliptic curve is an affine variety of dimension $g-1$.
\end{theorem}
From the deformation equation
\begin{equation}
\begin{aligned}
&\frac{\partial}{\partial\tau}\Ab(\Gamma_{\epsilon}^{\ast})_j\\
&=g_j(p)-\sum_{k=1}^{g}\sum_{n=1}^{\infty}\frac{\epsilon^n}{n!(n-1)!}\sum_{s+m=n-1}\begin{pmatrix}n-1\\m\end{pmatrix}(\frac{d^{n-1+m}}{dt^{n-1+m}}f^{\ast}_{j}(p_0))\frac{d^{s}}{dt^{s}}f_k(p_0)h_k(p),
\end{aligned}
\end{equation}
in order to have that $\frac{\partial}{\partial\tau}\Ab(\Gamma_{\epsilon}^{\ast})(p)$ is a rational vector for all $\epsilon$, we must have

\begin{equation}
\begin{aligned}
0=&\sum_{k=1}^{g}\sum_{n=1}^{\infty}\frac{\epsilon^n}{n!(n-1)!}\sum_{s+m=n-1}\begin{pmatrix}n-1\\m\end{pmatrix}(\frac{d^{n-1+m}}{dt^{n-1+m}}f^{\ast}_{j}(p_0))\frac{d^{s}}{dt^{s}}f_k(p_0)h_k(p)\\
&=\sum_{n=1}^{\infty}\frac{\epsilon^n}{n!(n-1)!}\sum_{s+m=n-1}\begin{pmatrix}n-1\\m\end{pmatrix}(\frac{d^{n-1+m}}{dt^{n-1+m}}f^{\ast}_{j}(p_0))\sum_{k=1}^{g}\frac{d^{s}}{dt^{s}}f_k(p_0)h_k(p).
\end{aligned}
\end{equation}
Hence,   it is an elliptic soliton if
\begin{equation}\label{el2}
\sum_{k=1}^{g}\frac{d^i}{dt^i}f_k(p_0)h_k(p)=0\text{ for }i=0,\dots,\infty.
\end{equation}
A similar argument in Section~\ref{hysch} shows that there is no nonzero vector $(c_1,\dots,c_g)$ such that
\begin{equation}
\sum_{k=1}^{g}c_k\frac{d^i}{dt^i}f_k(p_0)=0\text{ for }i=0,\dots,\infty.
\end{equation}
The other possibility of Equation~\eqref{el2} satisfied is that there exists a basis
$\{\eta_1,\dots,\eta_g\}$ such that $h_k(p)=0$ for $k=1,\dots g$. Consequently, if so, we see that any deformation of elliptic solitons under Schiffer variation at $p_0\ne p$ gives an elliptic soliton. But this contradicts to Theorem~\ref{th1} and Theorem~\ref{thel} unless $g=1$. Hence we prove

\begin{theorem}
A Schiffer variation can not preserve the deformation of elliptic solitons  unless $g=1$.
\end{theorem}

\begin{corollary}
Let $\mathscr{C}_g$ be the space of rank 1 transformations. Then
$$\delta(\TT_{[\Gamma]}\widetilde{\mathscr{E}_g})\bigcap\mathscr{C}_g=\emptyset\text{ for }g\geq5.$$
\end{corollary}

\end{document}